\documentclass[final,1p,times]{elsarticle}
\usepackage{amsthm}
\usepackage{amscd}
\usepackage{amsmath}
\usepackage{amsfonts}
\usepackage{amssymb}
\usepackage{graphicx}
\newtheorem{theorem}{Theorem}[section]

\newtheorem{lemma}[theorem]{Lemma}

\newtheorem{definition}{Definition}[section]
\usepackage{mathrsfs}
\usepackage{titletoc}
\numberwithin{equation}{section}
\newcommand{\rw}{\rightarrow}

\newcommand{\D}{\Delta}
\newcommand{\ra}{\rightarrow}

\newcommand{\f}{\frac}

\renewcommand{\l}{\lambda}

\newcommand{\be}{\begin{equation}}
\renewcommand{\ra}{\rightarrow}
\newcommand{\ee}{\end{equation}}
\newcommand{\bea}{\begin{eqnarray}}
\newcommand{\eea}{\end{eqnarray}}
\newcommand{\bna}{\begin{eqnarray*}}
\newcommand{\ena}{\end{eqnarray*}}

\renewcommand{\O}{\Omega}
\newcommand{\io}{\int_{\Omega}}
\renewcommand{\le}{\left}
\newcommand{\ri}{\right}

\newcommand{\ve}{\vert}
\newcommand{\V}{\Vert}

\newcommand{\ep}{\epsilon}

\newcommand{\na}{\nabla}

\journal{***}
\begin{document}

\begin{frontmatter}

\title{ Multiple solutions of a nonlinear biharmonic equation on graphs}

\author{Songbo Hou \corref{cor1}}
\ead{housb@cau.edu.cn}
\address{Department of Applied Mathematics, College of Science, China Agricultural University,  Beijing, 100083, P.R. China}

\cortext[cor1]{Corresponding author: Songbo Hou}

\begin{abstract}

In this paper, we  consider a  biharmonic equation with respect to the Dirichlet problem on a domain of a locally finite graph.
Using the variation method, we prove that the equation has two distinct solutions under certain conditions.

\end{abstract}


 \end{frontmatter}
 %\tableofcontents

\section{Introduction}

The existence and non-existence of solutions of boundary value problems for biharmonic equations has been studied by many authors. A lot of results are devoted to the  following problem in $H^2_0(\O)$,
\begin{equation}\label{eq1}
\left\{
\begin{aligned}
&\Delta^2u=\l u+\vert u\vert^{p-2}u,\,x\in\O\\
&u|_{\partial\O}=0,\,\f{\partial u}{\partial n}\bigg|_{\partial \O}=0,
\end{aligned}\right.
\end{equation}
where $\O$ is a bounded smooth domain in $\mathbb{R}^N$; $\l$ is a constant; $p=\f{2N}{N-4}$. If $\l<0$ and $\O$ is star-shaped, then the problem (1.1) has only trivial solution. Denote by $\l_1$ the first eigenvalue of the problem
\begin{eqnarray*}
\left\{
\begin{aligned}
&\Delta^2u=\l u,\,x\in\O\\
&u|_{\partial\O}=0,\,\f{\partial u}{\partial n}\bigg|_{\partial \O}=0.
\end{aligned}\right.
\end{eqnarray*}

If $0<\l<\l_1$ and $N\geq 8$, then Problem (1.1) has at least one nontrivial solution \cite{EFJ}. For $5\leq n\leq 7$, there exists a constant $\bar{\l}>0$ such that
 Problem (1.1) has a non-trivial solution for all $\l\in(\bar{\l},\l_1)$.
Deng and Wang \cite{DW99} studied the existence and non-existence of multiple solutions of biharmonic equations boundary value problem,
\begin{eqnarray*}
\left\{
\begin{aligned}
&\Delta^2u=\l u+\vert u\vert^{p-2}u+f(x),\,x\in\O,\\
&u|_{\partial\O}=0,\,\f{\partial u}{\partial n}\bigg|_{\partial \O}=0.
\end{aligned}\right.
\end{eqnarray*}
We refer the reader to \cite{AF13,CD19,WS2009,WZ08} for more related results.

In this paper,we will study the similar problem on a graph. A discrete graph is denoted by $G=(V,E)$,  where $V$ is the vertex set and $E$  is the edge set.
 We say that $G$ is  locally finite if  for any $x\in V$, there are only finite $y\in V$ such that $xy\in E$. A graph is called connected  if  for any $x,y\in V$,
they can be  connected via finite edges.     Throughout this paper, we assume that $G$ is locally finite and connected. Let $\omega_{xy}$ be  the weight of an edge $xy\in E$ such that $\omega_{xy}>0$ and $\omega_{xy}=\omega_{yx}$. We use a positive function $\mu:V\rw \mathbb{R}^+$ to define a measure on $G$. For a bounded domain $\O\subset V$, the boundary of $\O$ is defined as
$$\partial\O:=\{y\notin\O:\exists \,x\in\O\, \text{such that}\,xy\in E\}.$$
We introduce some notations about the partial differential equations on a graph. The $\mu-$ Laplacian of a function $u:V\rw \mathbb{R}$ is defined by
$$\D u(x):=\f{1}{\mu(x)}\sum\limits_{y\sim x}\omega_{xy}(u(y)-u(x)),$$ where $y\sim x$ stands for $xy\in E$. For any two functions $u$ and $v$ on the graph, the gradient form is defined by
$$\Gamma(u,v)=\f{1}{2\mu(x)}\sum\limits_{y\sim x}\omega_{xy}(u(y)-u(x)(v(y)-v(x)).$$
If $u=v$, we write $\Gamma(u)=\Gamma(u,u)$, which is used to define the length of the gradient for $u$

$$\vert \na u\vert(x) =\sqrt{\Gamma(u)(x)}=\Bigg(\f{1}{2\mu(x)}\sum\limits_{y\sim x}\omega_{xy}(u(y)-u(x))^2\Bigg)^{1/2}.$$
For  a function $u$ over $V$, the integral is defined by
$$\int_V ud\mu=\sum\limits_{x\in V}\mu(x)u(x).$$
Recently,   Grigor'yan, Lin and  Yang \cite{GLY16, GLY17} applied the mountain-pass theorem to establish existence results for some nonlinear equations.
 Zhang and Zhao \cite{ZZ18} studied the convergence of ground state solutions for nonlinear
Schr\"{o}dinger equations via the Nehari method on graphs. Using the similar method, Han, Shao and Zhao \cite{HSZ}  studied existence and convergence of solutions for nonlinear biharmonic equations on graphs. One may refer to \cite{LW17} for the results about the heat equation and refer to \cite{ GJ18, GLY, LY20} for the results about the Kazdan-Warner equation on a graph.

 In this paper, we  consider the following equation
 \be\label{bhe}
\left\{
\begin{aligned}
&\Delta^2u=\l u+\ve u\ve^{p-2}u+\ep f,&&\text{in}&&\O,\\
&u=0,&&\text{on}&&{\partial\O},
\end{aligned}\right.
\ee
where $\O$ is a bounded domain of $V$; $\l$, $p$ and $\ep$ are positive constants, $p>2$; $f$ is a given function on $V$.

Let $W^{2,2}(\O)$ be the space of functions $u:V\rw \mathbb{R}$ under the norm

\be\label{norm}\V u\V_{W^{2,2}(\O)}=\left(\int_{\O\cup\partial \O}(\ve \D u\ve^2+\ve \na u\ve^2)d\mu+\io u^2d\mu\right)^{\f{1}{2}}.\ee

 Let $H(\O)=W^{2,2}(\O)\cap W_0^{1,2}(\O)$, where $W_0^{1,2}(\O)$ is the completion of $C_c(\O)$ under the norm   $$ \label{nor}\V u\V_{W^{1,2}_0(\O)}=\left(\int_{\O\cup\partial \O}\ve \na u\ve^2d\mu+\io u^2d\mu\right)^{\f{1}{2}}.$$

 Define
$$\V u\V_H=\left(\int_{\O\cup\partial \O}\ve \D u\ve^2d\mu\right)^{\f{1}{2}}$$ for any $u\in H$. It is easy to see that $\V \cdot\V_H$ is a norm on $H$.
Noting that the dimension of $W^{2,2}(\O)$ is finite, we have that
$\V u\V_H=\left(\int_{\O\cup\partial \O}\ve \D u\ve^2d\mu\right)^{\f{1}{2}}$ is a norm equivalent to (\ref{norm}) on $H(\O)$.
Lemma 2.6 in \cite{HSZ} implies that  $H(\O)$ is embedded in $L^q(\O)$ for all $1\leq q<+\infty$ and there is a constant depending only on $q$ and $\O$ such that
\be \le( \io \ve u\ve^q d\mu\ri)^{1/q}\leq C\left(\int_{\O\cup\partial \O}\ve \D u\ve^2d\mu\right)^{\f{1}{2}}.\ee
 \begin{definition}
For a $u\in H$, if for any $\varphi\in C_c(\O)$, there holds that
$$\int_{\O\cup\partial\O}\D u\D\varphi d\mu=\l\int_{\O}u\varphi d\mu+\int_{\O}\vert u\vert^{p-2} u\varphi d\mu+\ep\int_{\O}f\varphi d\mu,$$
then $u$ is called a weak solution of (\ref{bhe}).
\end{definition}

Define
$$J_{\ep}(u)=\f{1}{2}\int_{\O\cup\partial \O}\ve \D u\ve^2d\mu-\f{1}{2}\io\l u^2d\mu-\f{1}{p}\io \ve u\ve^pd\mu-\ep\io f(x)ud\mu,$$ which will be used in the later variation procedure.
Define $$\l_1(\O)=\inf\limits_{u\nequiv 0,\,u\vert_{\partial \O}=0}\f{\int_{\O\cup\partial \O}\vert \D u\vert^2 d\mu}{\io u^2d\mu}.$$
 Using the variation method similar to that in \cite{GLY16,GLY17}, we prove the following theorem.
\begin{theorem} Let $G=(V,E)$ be a locally finite graph. Suppose that $0<\l<\l_1(\O)$, $f\in H'(\O)$  where $H'$ is the dual space of $H$.  Then there exists $\ep_1>0$ such that
(\ref{bhe}) has two distinct solutions if $0<\ep<\ep_1$.
\end{theorem}

\section {Proof of the main results }
\begin{lemma}
There exist  positive constants $r_{\ep}$ and $\delta_{\ep}$ such that $J_{\ep}\geq \delta_{\ep}$ for all $u\in H$ with $r_{\ep}\leq\V u\V_{H}\leq 2 r_{\ep}$ if $0<\ep<\ep_1$ for a sufficiently small $\ep_1$.

\end{lemma}
\begin{proof}
\be \label{gi1}\begin{split}
J_{\ep}(u)\geq &\f{1}{2}\V u\V_{H}^2-\f{\l}{2}\io\ve u\ve^2d\mu-\f{C}{p}\V u\V_H^{p}-\ep \V f\V_{H^{\prime}}\V  u\V_{H}\\
\geq & \f{\tau}{2} \V u\V_{H}^2-\f{C}{p}\V u\V_H^{p}-\ep \V f\V_{H^{\prime}}\V  u\V_{H}\\
\geq & \V u\V_{H}\le( \f{\tau}{2}\V u\V_{H}-\f{C}{2}\V u\V^{p-1}-\ep \V f\V_{H^{\prime}}\ri),
\end{split}
\ee
where $\tau=\f{\l_1(\O)-\l}{\l_1(\O)}$.
Take $r_{\ep}=\sqrt{\ep}$.
 Since $$\lim\limits_{\ep\rightarrow 0^+}\f{ \f{\tau}{2}\sqrt{\ep}-2^{p-2}C\ep^{(p-1)/2}-\ep \V f\V_{H^{\prime}}}{\f{\tau}{2}\sqrt{\ep}}=1, $$ there exists some $\ep_1$ such that $$\f{\tau}{2}\sqrt{\ep}-2^{p-2}C\ep^{(p-1)/2}-\ep \V f\V_{H^{\prime}}\geq \frac{\tau}{4} \sqrt{\ep} $$ if $0<\ep <\ep_1$.

   Setting $\delta_{\ep}=\f{\tau\ep}{4}$,  we obtain $J_{\ep}(u)\geq \delta_{\ep}$ if $0<\ep<\ep_1$.
\end{proof}
\begin{lemma}
For any $c\in \mathbb{R}$, $J_{\ep}$ satisfies the $(\text{PS})_c$ condition. If $(u_k)\subset H $ is a sequence such that  $J_{\ep}(u_k)\rightarrow c$ and $J_{\ep}^{\prime}(u_k)\rightarrow 0$, then up to a subsequence, $u_k$ converges  to some $u$ in $H$.
\end{lemma}
\begin{proof}
If $J_{\ep}(u_k)\rightarrow c$ and $J_{\ep}^{\prime}(u_k)\rightarrow 0$, then we have
\be\label{fu1}\f{1}{2}\int_{\O\cup\partial \O}\ve \D u_k\ve^2d\mu-\f{1}{2}\io\l u_k^2d\mu-\f{1}{p}\io \ve u_k\ve^pd\mu-\ep\io f(x)u_kd\mu=c+o_k(1),\ee
\be\label{fu2} \le|\int_{\O\cup\partial \O}\ve \D u_k\ve^2d\mu-\io\l u_k^2d\mu-\io \ve u_k\ve^pd\mu-\ep\io f(x)u_kd\mu\ri|= o_k(1)\V u_k\V_H,\ee
where $o_k(1)\rightarrow 0$ as $k\rightarrow +\infty$.

From (\ref{fu1}) and (\ref{fu2}), we have
\be \le(\f{1}{2}-\f{1}{p}\ri)\io \ve u_k\ve^pd\mu=c+\f{\ep}{2}\io fu_kd\mu+o_k(1)\V u_k\V_H+o_k(1).\ee
Hence
 \be \begin{split}
\tau \V u_k\V^2_H \leq &\int_{\O\cup\partial \O}\ve \D u_k\ve^2d\mu-\io\l u_k^2d\mu\\
\leq & \f{2pc}{p-2}+\f{(2p-2)\ep}{p-2}\V f\V_{H^{\prime}}\V u_k\V_{H}+o_k(1)\V u_k\V_H+o_k(1)\\
\leq & \f{2pc}{p-2}+\f{4(p-1)^2\ep^2}{(p-2)^2\tau}\V f\V_{H^{\prime}}^2+\f{\tau}{4} \V u_k\V^2_{H}+\f{\tau}{4}\V u_k\V^2_{H}+o_k(1).
\end{split}
\ee
Thus we have $(u_k)$ is bounded in $H$. Since $H$ is pre-compact, it follows that up to a subsequence, $u_k$ converges  to some $u$ in $H$.
\end{proof}
Now we arrive at a position to prove the main theorem. For any $u^*\in H$, passing to the limit $t\rightarrow +\infty$, we get
$$J_{\ep}(tu^*)=\f{t^2}{2}\int_{\O\cup\partial \O}\ve \D u^*\ve^2d\mu-\f{t^2}{2}\io\l {u^*}^2d\mu-\f{t^p}{p}\io \ve u^*\ve^pd\mu-t\ep\io f(x)u^*d\mu\rightarrow -\infty$$ as $t\rw \infty$.
Hence there exists some $\tilde{u}\in H$ such that
$J_{\ep}(\tilde{u})<0$ with $\V \tilde{u}\V_H>r_{\ep}$. Combining Lemma 2.1, we see that $J_{\ep}$ satisfies all the hypotheses of the mountain-pass theorem:
 $J_{\ep}\in C^1(H,\mathbb{R}) $; $J_{\ep}(0)=0$; when $\V u\V_{H}=r_{\ep}$, $J_{\ep}(u)\geq \delta_{\ep}$; $J_{\ep}(\tilde{u})<0$ for some $\tilde{u}$ with $\V \tilde{u}\V_H>r_{\ep}$.

Then we have
$$c=\min\limits_{\gamma\in \Gamma}\max\limits_{u\in \gamma}J_{\ep}(u)$$ is the critical point of $J_{\ep}$, where
$$\Gamma=\{\gamma\in C([0,1], H):\gamma (0)=0, \gamma (1)=\tilde{u}\}.$$

Thus there exists a weak solution $u_c\in H$ with $J_{\ep}(u_c)\geq \delta_{\ep}$.

\begin{lemma}
There exists $\tau_0$ and $u^*\in H$ with $\V u^*\V_H=1$ such that $J_{\ep}(tu^*)<0$  if  $0<t<\tau_0$.
\end{lemma}
\begin{proof}
We study the equation
\be\label{ne2}\Delta^2u=\l u+  f\ee in $H(\O)$.  Define the functional
$$J_{f}(u)=\f{1}{2}\int_{\O\cup\partial \O}\ve \D u\ve^2d\mu-\f{1}{2}\io\l u^2d\mu-\io f(x)ud\mu. $$

Noting that \be\label{je1}\f{1}{2}\int_{\O\cup\partial \O}\ve \D u\ve^2d\mu-\f{1}{2}\io\l u^2d\mu\geq \f{\tau}{2}\V u\V_{H}^2 \ee
and \be\label{je2}\le|\io f(x)ud\mu\ri|\leq \V f\V_{H^{\prime}}\V u\V_{H}\leq \f{\tau}{4}\V u\V_{H}^2+\f{1}{\tau}\V f\V_{H^{\prime}}^2,\ee
we  have that
$J_{f}$ have a lower bound on $H$.

Set
$$m_f=\inf\limits_{u\in H}J_f(u).$$  There exists $u_k\in H$ such that $J_f(u_k)\ra m_f$.  From  (\ref{je1}) and (\ref{je2}), we know $u_k$ is bounded in $H$.

Hence there exists $u^*\in H$ such that $u_k\rightharpoonup \bar{u}$ weakly in $H$. Then
$$J_f(\bar{u})\leq \lim\inf\limits_{k\rightarrow\infty}J_f(u_k)=m_f$$ and $\bar{u}$ is the weak solution of (\ref{ne2}).
It follows that
\be\label{nq2}\io f\bar{u}d\mu=\int_{\O\cup\partial \O}\ve \D \bar{u}\ve^2d\mu-\l\io (\bar{u})^2d\mu>0\ee
Now we compute the derivative of $J_{\ep}(t\bar{u})$:
$$\f{d}{dt}J_{\ep}(t\bar{u})=t\int_{\O\cup\partial \O}\ve \D \bar{u}\ve^2d\mu-t\io\l (\bar{u})^2d\mu-t^{p-1}\io \ve \bar{u}\ve^pd\mu-\ep\io f(x)\bar{u}d\mu.$$

By (\ref{nq2}), we obtain
$$\f{d}{dt}J_{\ep}(t\bar{u})\Big\ve_{t=0}<0.$$
Letting $u^*=\f{\bar{u}}{\V\bar{u}\V_{H}}$, we  finish the proof.
\end{proof}
\begin{lemma}
Choose $\ep$ such that  $0<\ep<\ep_1$ where $\ep_1$ is the same as in Lemma 2.1. Then there exists a function $u_0\in H$ with $\V u_0\V_{H}\leq 2r_{\ep}$ such that
$$J_{\ep}(u_0)=c_{\ep}=\inf\limits_{\V u\V\leq 2r_{\ep  }}J_{\ep}(u),$$
where $r_{\ep}=\sqrt{\ep},$ $c_{\ep}<0$.
\end{lemma}
\begin{proof}
By (\ref{gi1}), we see that $J_{\ep}$ has a lower bound on
$$B_{2r_{\ep}}=\{u\in H:\V u\V_{H}\leq 2r_{\ep}\}.$$
Combining Lemma 2.3, we get
$c_{\ep}<0$.

Let $(u_{k})\subset H$ be a sequence satisfying  $\V u_k\V_{H}\leq 2r_{\ep}$ and $J_{\ep}(u_k)\rw c_{\ep}$. Then up to a sequence, $u_k$ converges weakly to $u_0$  in $H$ and converges strongly to $u_0$  in $L^q(V)$ for any $1\leq q\leq +\infty$. It follows that

$$\lim\limits_{k\rw +\infty}\int_{\O} fu_kd\mu=\int_{\O} fu_0 d\mu,$$
$$\V u_0\V_H\leq \limsup\limits_{k\rw +\infty}\V u_k\V_H\leq 2r_{\ep},$$
$$\lim\limits_{k\rw +\infty}\int_{\O}\l u_k^2d\mu=\int_{\O}\l u_0^2d\mu,$$
$$\lim\limits_{k\rw +\infty}\int_{\O} u_k^pd\mu=\int_{\O} u_0^p d\mu.$$
Hence $$J_{\ep}(u_0)\leq \limsup\limits_{k\rw +\infty}J_{\ep}(u_k)=c_{\ep}$$ and $u_0$ is the minimizer of $J_{\ep}$ on $B_{2r_{\ep}}$. Lemma 2.1 implies that $\V u_0\V_H<r_{\ep}$.
For any $\varphi \in C_c(V)$, let $\psi(t)=J_{\ep}(u_0+t\varphi)$. Then $\psi$ is a smooth function in $t$.
It is easy to see that there is $\eta>0$ such that $u_0+t\varphi\in B_{2r_{\ep}}$ if $\vert t\vert <\eta$. This means that $\psi(0)$ is the minimum of $\psi(t)$ on $(-\eta,\eta)$. By $\psi'(0)=0$, we get
$$\int_{\O\cup\partial \O}\Delta u_0\Delta \varphi d\mu-\l\int_{\O} u_0\varphi d\mu-\int_{\O}\vert u_0\vert ^{p-1} u_0\varphi d\mu-\ep\int_{\O} f\varphi d\mu=0.$$

We conclude that $u_0$ is a weak solution of (\ref{bhe}) and complete the proof.
\end{proof}
Clearly, $u_c$ and $u_0$ are two distinct solutions of (\ref{bhe}) since $J_{\ep}(u_c)>0$ and $J_{\ep}(u_0)<0$.

\vskip 30 pt
\noindent{\bf Acknowledgement}

This work is  partially  supported by the National Natural Science Foundation of China (Grant No. 11721101), and by National Key Research and Development Project SQ2020YFA070080.

\end{document}